\documentclass[10pt]{amsart}
\usepackage{graphicx}
\usepackage{stix}
\usepackage[margin=1in]{geometry}
\usepackage{color}
\usepackage{hyperref}
\usepackage{lineno}
\newtheorem{definition}{Definition}[section]

\newtheorem{lemma}{Lemma}[section]
\newtheorem{corollary}{Corollary}[section]
\newtheorem{remark}{Remark}[section]
\newtheorem{theorem}{Theorem}[section]
\numberwithin{equation}{section}
\begin{document}
\title[Stability for evolution equations with variable growth]{Stability for evolution equations with variable growth}
\author[S.~Shmarev, J.~Simsen, M.~Stefanello Simsen]{Sergey Shmarev \and Jacson Simsen \and Mariza Stefanello Simsen
\address{S.~Shmarev: Depto.~Matemáticas, Universidad de Oviedo, c/Federico García Lorca 18, 33007 Espa\~na}
\email{shmarev@uniovi.es}
\address{J.~Simsen: Instituto de Matemática e Computa\c{c}\~{a}o, Universidade Federal
de Itajubá, 37500-903 Itajubá, Minas Gerais, Brazil}
\email{jacson@unifei.edu.br}
\address{M.~Stefanello Simsen: Instituto de Matemática e Computa\c{c}\~{a}o, Universidade Federal
de Itajubá, 37500-903 Itajubá, Minas Gerais, Brazil}
\email{mariza@unifei.edu.br}}
\thanks{The first authors acknowledges the support of the Research Grant MTM2017-87162-P, España}


\date{20.03.2021}

\begin{abstract}
We study the character of dependence on the data and the nonlinear structure of the equation for the solutions of the homogeneous Dirichlet problem for the evolution $p(x,t)$-Laplacian with the nonlinear source
\[
u_t-\Delta_{p(x,t)}u=f(x,t,u),\quad (x,t)\in Q=\Omega\times (0,T),
\]
where $\Omega$ is a bounded domain in $\mathbb{R}^n$, $n\geq 2$, and $p(x,t)$ is a given function $p(\cdot):Q\mapsto (\frac{2n}{n+2},p^+]$, $p^+<\infty$. It is shown that the solution is stable with respect to perturbations of the variable exponent $p(x,t)$, the nonlinear source term $f(x,t,u)$, and the initial data. We obtain quantitative estimates on the norm of the difference between two solutions in a variable Sobolev space through the norms of perturbations of the nonlinearity exponent and the data $u(x,0)$, $f$. Estimates on the rate of convergence of a sequence of solutions to the solution of the limit problem are derived.
\end{abstract}

%
%
%



\maketitle

\section{Introduction}
\label{sec:Introduction}
The paper addresses the question of continuous dependence on the data for solutions of the Dirichlet problem for the quasilinear parabolic equations with variable nonlinearity:

\begin{equation}
\label{eq:prot}
\begin{cases}
& u_t-\Delta_{p(z)}u=f(z,u)\quad \text{in $Q$},
\\
& \text{$u=0$ on $\partial \Omega\times (0,T)$},
\\
& \text{$u(x,0)=u_0(x)$ in $\Omega$},
\end{cases}
\end{equation}
where $\Omega\subset \mathbb{R}^n$, $n\geq 2$, is a bounded domain, $Q=\Omega\times (0,T)$ is the cylinder of the finite height $T$. By $z=(x,t)$ we denote the points of $Q$. The differential operator

\[
\Delta_{p(z)}u=\operatorname{div}\left(|\nabla u|^{p(z)-2}\nabla u\right)
\]
with a given function $p:Q\mapsto (1,\infty)$ is a generalization of the classical $p$-Laplacian operator with constant $p$. The right-hand side of equation \eqref{eq:prot} is a given function of its arguments. We will distinguish between the cases where $f$ is a given function of the independent variables $z\in Q$, or depends also on the solution $u(z)$. In the latter case, we consider either the nonlinear sources of the form

\begin{equation}
\label{eq:source-intro}
f(z,s)=-a|s|^{\sigma(z)-2}s+f_0(z)
\end{equation}
with a given function $\sigma:Q\mapsto [2,\infty)$, $a=const$, and $f_0$ in a suitable Lebesgue space, or $f(z,u)=\phi(u)+f_0(z)$ with a Lipschitz-continuous function $\phi(\cdot)$.

By now, the theory of PDEs with variable nonlinearity accounts for numerous results on the issues of existence, uniqueness, and qualitative properties of solutions. The results on the character of dependence of solutions to problem \eqref{eq:prot} on the data and on the nonlinear structure of the equations are scarce, albeit the study of the stability of the mathematical models based on PDEs is important to ensure applicability of the theory to real-world problems because the mathematical models are built, in general, on approximate experimental data.

It is known that the solutions of the Dirichlet problem for the evolution $p(z)$-Laplace equation and the source term with $\sigma(z)=p(z)$, $f_0=0$ and $a\geq 0$ are stable with respect to the initial data \cite{Erhardt-2017-stab}: the solutions of problem \eqref{eq:prot} corresponding to the initial data $u_0$, $v_0$ satisfy the estimate

\[
\|u(\cdot,t)-v(\cdot,t)\|_{2,\Omega}^{2}\leq \|u_0-v_0\|_{2,\Omega}^{2}\qquad \text{for a.e. $t\in (0,T)$}.
\]
Similar stability estimates in $L^2(\Omega)$ were proven in \cite{Erhardt-2021} for the solutions of pa\-ra\-bo\-lic systems with nonstandard growth and a cross-diffusion term. The $p(z)$-Laplacian is a prototype of the operators with nonstandard growth considered in \cite{Erhardt-2017-stab,Erhardt-2021}. In \cite{Huashui-2021}, the stability estimates in $L^1(\Omega)$ with respect to the initial data were derived for the solutions of anisotropic parabolic equations with double variable nonlinearity, convective terms, and possible degeneracy on the lateral boundary of the problem domain.

Continuity of solutions of equation \eqref{eq:prot} with respect to the variable exponent $p$ and convergence to the solution of the limit problem was discussed in papers \cite{SSP1,SS,SSP2,SSP3}. In these works, continuity in $C([0,T];L^2(\Omega))$ is proven for the solutions of degenerate equations with $p\equiv p(x)>2$. These restrictions are due to the method of the study, based on the analysis of the semigroup generated by the operator $-\Delta_{p(\cdot)}$ in $L^2(\Omega)$.

Another approach to this problem was developed in \cite{Kinnunen-Parviainen-2010,Lukkari-Parviainen-2015}. In \cite{Kinnunen-Parviainen-2010}, continuity of solutions of the evolution $p$-Laplace equation with respect to $p$ was considered for the Dirichlet problem in a cylinder, and in \cite{Lukkari-Parviainen-2015} for the Cauchy problem with constant $p$. It is shown in \cite{Kinnunen-Parviainen-2010} that the solutions to the problem
\begin{equation}
\label{eq:p-const}
\text{$u_t-\Delta_pu=0$ in $Q$},\qquad \text{$u=\phi$ on the parabolic boundary of $Q$}
\end{equation}
with constant $p\geq 2$ and a sufficiently regular boundary $\partial\Omega$ are continuous with respect to the perturbation of the exponent $p$: let $p_i\to p$ and $u_i$ be the corresponding solutions of problem \eqref{eq:p-const}; there exists $\epsilon>0$ such that $u_i\to u$ in $L^{p+\epsilon}(0,T;W^{1,p+\epsilon}(\Omega))$ where $u$ is the solution of the limit problem. A specific difficulty is that the solutions corresponding to different $p_i$ belong to different energy spaces prompted by \eqref{eq:p-const}:

\[
\int_{Q}\left(|u_i|^{p_i}+|\nabla u_i|^{p_i}\right)\,dz<\infty.
\]
This difficulty is overcome due to the property of global higher integrability of the gradient. It turns out that there exists $\epsilon>0$ such that $|\nabla u_i|\in L^{p+\epsilon}(Q)$ for the sufficiently large $i$, which allows one to conclude that the sequence $\{u_i\}$ is a Cauchy sequence in this space and converges to the solution of the limit problem.

In the stationary case, it is shown in \cite{Eleuteri-2013} that the solutions of the obstacle problem for elliptic equations with nonstandard growth are stable with respect to perturbations of their nonlinear structure. The $p(x)$-Laplace equation is a prototype of the class of equations studied in \cite{Eleuteri-2013}, where the property of global higher integrability of the gradient is also one of the main ingredients in the proof of stability.

In the present paper, we are interested in the continuity of solutions of problem \eqref{eq:prot} with respect to the initial data $u_0$, the source $f$ in the form \eqref{eq:source-intro}, and the exponents of nonlinearity $p(z)$, $\sigma(z)$. We find sufficient conditions of continuity and estimate the moduli of continuity with respect to each component of the vector $\{p(z),u_0,f_0,\sigma(z)\}$ (the data). The proof of continuity with respect to $p(z)$ relies on the property of global higher integrability of the gradient. This property is derived in \cite{ArShm-2021} for the solutions of problem \eqref{eq:prot} with $f(z,u)\equiv f(z)$ under certain assumptions on the regularity of the variable exponent $p(z)$ and the data $u_0$, $f$, which allow one to show also that the corresponding solution possesses better regularity than the weak solutions. Such solutions are called the strong solutions - see Definition \ref{def:strong-sol}. The results on continuity are formulated in terms of the norm in a variable Sobolev space of the difference between a strong solution corresponding to a regular exponent $p(z)$ and the weak solution of the same problem with an exponent $q(z)$.

Organization of the paper. In Section \ref{sec:prelim} we introduce the variable Lebesgue and Sobolev spaces and collect their basic properties. Section \ref{sec:p-Laplace} is devoted to the study of problem \eqref{eq:prot} with the  independent of $u$ source term $f(z)$. The main result is given in Theorem \ref{th:stab-1}. Let us illustrate it by the following example. Assume that

\[
\partial\Omega\in C^2,\quad p(\cdot),q(\cdot):Q\mapsto\left(\dfrac{2n}{n+2}, \beta\right),\;\; \beta<\infty,\quad f\in L^2(0,T;W^{1,2}_0(\Omega))
\]
and

\[
\int_{\Omega}|\nabla u_0|^{s(x)}\,dx<\infty\quad \text{with  $s(x)=\max\{2,p(x)\}$}.
\]
If $p(\cdot)$ is Lipschitz-continuous in $\overline{Q}$, $q(\cdot)$ is continuous in $\overline{Q}$ with a logarithmic modulus of continuity, and $-\alpha\leq p(z)-q(z)\leq \gamma$ with positive constants $\alpha,\gamma$ depending only on $n$ (condition \eqref{eq:proximity}), then the solutions of problem \eqref{eq:prot} $u$, $v$ with the exponents $p$ and $q$, and the same initial datum and the source term, satisfy the inequality

\begin{equation}
\label{eq:stab-example}
\|u-v\|_{L^\infty(0,T;L^2(\Omega))}^2+\int_Q|\nabla u-\nabla v|^{q(z)}\,dz\leq C\left(\mathcal{R}+\mathcal{R}^{\frac{q^+}{2}}+\mathcal{R}^{\frac{q^-}{2}}\right)
\end{equation}
where
\[
\mathcal{R}=\sup_Q\left(|p-q|^{\frac{q}{q-1}}\right),\quad q^+=\sup_Q q,\quad q^-=\inf_Q q.
\]
The constant $C$ depends on the structural constants and the norms of $u$ and $v$ in the corresponding energy spaces, which are estimated through the data. It is worth noting that we do not distinguish between the degenerate and singular equations, i.e., $p(z),q(z)> 2$, or $1<p(z),q(z)< 2$. Moreover, the regions of degeneracy and singularity of the equations for $u$ and $v$ may overlap but need not coincide. In the general case where the solutions $u$ and $v$ correspond to different initial data $u_0$, $v_0$ and the source terms $f_1$, $f_2$, $\mathcal{R}$ depends also on $\|u_0-v_0\|_{2,\Omega}$ and $\|f_1-f_2\|_{2,Q}$. In Theorem \ref{th:comp-p-Laplace} we present results on convergence of families of strong solutions of problem \eqref{eq:prot} to weak (strong) solutions of the limit problem.

In Section \ref{sec:pert} we consider problem \eqref{eq:prot} with a  nonlinear source. The main attention is paid to the case of the source of the form \eqref{eq:source-intro} with $a\geq 0$. The study is confined to the class of exponents $\sigma(z)$ for which the existence of a weak solution $u(z)$ is already known. This fact allows one to consider the nonlinear source $f(z,u(z))$ as a given function of the independent variable $z$ and apply the results of Section \ref{sec:p-Laplace}, which is possible for degenerate equations under the following assumptions on the range of the exponents in \eqref{eq:prot}, \eqref{eq:source-intro}:

\[
\text{$2\leq p(z)$,\quad $2\leq \sigma(z)\leq 1+\dfrac{p(z)}{2}$ in $\overline{Q}$.}
\]
We derive estimates of the type \eqref{eq:stab-example} for the solutions of problem \eqref{eq:prot}, \eqref{eq:source-intro} corresponding to the data $\{p(z), u_0(x), \\ f_0(z), \sigma(z)\}$ and $\{q(z),v_0(x),g_0(z),\mu(z)\}$ with $\mathcal{R}$ depending also on $\sup_Q|\sigma(z)-\mu(z)|$ and the norms of $u_0-v_0$, $f_0-g_0$. This result is true if at least one of the sets of data produces a strong solution.

In Subsection \ref{subsec:reaction} we extend the results to the nonlinear sources of the form $f(z,s)=\phi(s)+f_0(z)$ with the same  Lipschitz-continuous, not necessarily sign-definite function $\phi$. Since $\phi(s)$ is no longer assumed to be sign-definite, the norm of the difference between two solutions is estimated by the same quantity as in the case \eqref{eq:source-intro} with $a\geq 0$, but with a coefficient that grows exponentially in time.

\section{Preliminaries}
\label{sec:prelim}
We collect here the background information on the variable Lebesgue and Sobolev spaces used throughout the paper. We refer to the monograph \cite{DHHR-2011} for further information, and also to \cite[Ch.1]{ant-shm-book-2015} and \cite{DNR-2012} for the properties of spaces of functions defined on cylinders.

Let $\Omega$ be a bounded domain with Lipschitz-continuous boundary
$\partial \Omega$ and $p: \Omega \to [p^-, p^+] \subset (1,
\infty)$ be a measurable function. Define the functional

\[
\rho_{p(\cdot)}(f)=
\int_{\Omega} |f(x)|^{p(x)} \,dx
\]
(the modular). The set

\[
L^{p(\cdot)}(\Omega) = \{f:\Omega
\to \mathbb{R}:\text{$f$ is measurable on $\Omega$}, \rho_{p(\cdot)}(f) <
\infty\}
\]
equipped with the Luxemburg norm

\[
\|f\|_{p(\cdot),\Omega}= \inf \left\{\lambda>0 :
\rho_{p(\cdot)}\left(\dfrac{f}{\lambda}\right) \leq 1\right\}
\]
is a reflexive and separable Banach space and $C_0^\infty(\Omega)$ is dense in $L^{p(\cdot)}(\Omega)$. The modular $\rho_{p(\cdot)}(f)$ is lower semicontinuous.
By the definition of the norm

\begin{equation}
\label{eq:mod}
\min\{\|f\|^{p^-}_{p(\cdot), \Omega}, \|f\|^{p^+}_{p(\cdot), \Omega}\}
\leq \rho_{p(\cdot)}(f) \leq \max\{\|f\|^{p^-}_{p(\cdot), \Omega}, \|f\|^{p^+}_{p(\cdot), \Omega}\}.
\end{equation}
The dual of $L^{p(\cdot)}(\Omega)$ is the space $L^{p'(\cdot)}(\Omega)$ with the conjugate exponent  $p'=\dfrac{p}{p-1}$. For $f \in L^{p(\cdot)}(\Omega)$ and $g \in L^{p'(\cdot)}(\Omega)$, the generalized
H\"older inequality holds:

\begin{equation}
\label{eq:Holder}
\int_{\Omega} |fg| \leq
\left(\frac{1}{p^-} + \frac{1}{(p')^-} \right) \|f\|_{p(\cdot),
\Omega} \|g\|_{p'(.), \Omega} \leq 2 \|f\|_{p(.), \Omega}
\|g\|_{p'(\cdot), \Omega}.
\end{equation}
Let $p_1, p_2$ be two bounded
measurable functions in $\Omega$ such that $1<p_1(x) \leq p_2(x)$ a.e. in $\Omega$. Then $L^{p_2(\cdot)}(\Omega)$ is continuously
embedded in $L^{p_1(\cdot)}(\Omega)$ and

\[
\forall \,u \in
L^{p_2(\cdot)}(\Omega)\qquad \|u\|_{p_1(\cdot), \Omega} \leq C(|\Omega|,
p_1^\pm, p_2^\pm) \|u\|_{p_2(\cdot), \Omega}.
\]
The variable exponent Sobolev space $W_0^{1,p(\cdot)}(\Omega)$ is defined as the set of functions

\[
W_0^{1,p(\cdot)}(\Omega)= \{u: \Omega \to
\mathbb{R}\ |\  u \in L^{p(\cdot)}(\Omega) \cap W_0^{1,1}(\Omega),
|\nabla u| \in L^{p(\cdot)}(\Omega) \}
\]
equipped with the norm

\[
\|u\|_{W_0^{1,p(\cdot)}(\Omega)}= \|u\|_{p(\cdot),\Omega} +
\|\nabla u\|_{p(\cdot), \Omega}.
\]
Let $p\in C_{{\rm log}}(\overline{\Omega})$, {\it i.e.}, the exponent $p$ is continuous in $\overline{\Omega}$ with the logarithmic modulus of continuity:

\begin{equation}
\label{eq:log-cont}
|p(x_1)-p(x_2)| \leq \omega(|x_1-x_2|),
\end{equation}
where $\omega(\tau)$ is a nonnegative function satisfying the condition

\[
\limsup_{\tau \to 0^+} \omega(\tau) \ln
\left(\frac{1}{\tau}\right)= C< \infty.
\]
Then $C_0^{\infty}(\Omega)$ is dense in
$W_0^{1,p(\cdot)}(\Omega)$ and the Poincar\'e inequality holds: for all $u\in W_{0}^{1,p(\cdot)}(\Omega)$

\[
\|u\|_{p(\cdot),\Omega}\leq C\|\nabla u\|_{p(\cdot),\Omega}
\]
with a constant $C$ independent of $u(x)$.
The dual of $W_0^{1,p(\cdot)}(\Omega)$, denoted $W'(\Omega)$, is the set of bounded linear functionals over
$W_{0}^{1,p(\cdot)}(\Omega)$: $\Phi \in W'(\Omega)$ iff there
exist $\Phi_0 \in L^{p'(\cdot)}(\Omega), \Phi_i \in L^{p'(\cdot
)}(\Omega)$, $i=1,\ldots,N$, such that for all $u \in
W_0^{1,p(\cdot)}(\Omega)$

\[
\langle \Phi, u \rangle= \int_{\Omega} \left( u \Phi_0 +
\sum_{i=1}^{N} u_{x_i} \cdot \Phi_i\right) ~dx .
\]

For the study of parabolic problems we need the spaces of functions depending on $z\equiv (x,t)\in Q$. Let us define the spaces

\[
\begin{split}
 & V_t(\Omega) = \{u: \Omega \to \mathbb{R}\ |\  u \in L^2(\Omega)
\cap W_0^{1,1}(\Omega), |\nabla u|^{p(x,t)} \in L^{1}(\Omega)
\},\;\; t\in (0,T),
\\
& \mathbf{W}_{p(\cdot)}(Q)= \{u : (0,T) \to V_t(\Omega) \ |\ u \in L^2(Q), |\nabla
u|^{p(x,t)} \in L^1(Q)\}.
\end{split}
\]
The space $\mathbf{W}_{p(\cdot)}(Q)$ is equipped with norm

\[
\|u\|_{\mathbf{W}_p(\cdot)(Q)}=\|u\|_{2,Q}+\|\nabla u\|_{p(\cdot),Q}.
\]
The dual $\mathbf{W}'_{p(\cdot)}(Q)$ of the space $\mathbf{W}_{p(\cdot)}(Q)$ is the set of bounded linear functionals over $\mathbf{W}_{p(\cdot)}(Q)$:
$\Phi \in \mathbf{W}'_{p(\cdot)}(Q)$ iff there exists $\Phi_0\in L^{2}(Q)$, $\Phi_i \in
L^{p'(\cdot)}(Q)$, $i=1,\ldots,n$,  such that for all $u \in
\mathbf{W}_{p(\cdot)}(Q)$

\[
\langle \Phi, u \rangle= \int_{Q} \left( u \Phi_0 +
\sum_{i=1}^{n}u_{x_i} \Phi_i\right) ~dz.
\]
Let $C_{{\rm log}}(\overline{Q})$ be the set of functions satisfying condition \eqref{eq:log-cont} in the closure of the cylinder $Q$. If $u\in \mathbf{W}_{p(\cdot)}(Q)$, $u_t\in \mathbf{W}'_{p(\cdot)}(Q)$, and $p(z)\in C_{{\rm log}}(\overline{Q})$, then

\begin{equation}
\label{eq:int-by-parts}
\int_{Q}u u_t\,dz = \frac{1}{2} \int_{\Omega}u^{2}(x,t)\,dx\Big|_{t=0}^{t=T}.
\end{equation}

By convention, throughout the text $C$ denotes the constants which can be computed or estimated through the data but whose exact values are  unimportant. The symbol $C$ may be used for different constants inside the same formula, and may even vary from line to line.

Whenever it does not cause a confusion, we omit the arguments of the variable exponents of nonlinearity.

We repeatedly use the Young inequality in the form: for all $a, b>0$ and $r>1$
\[
a\cdot b \leq \dfrac{1}{r}a^r+\dfrac{1}{r'}b^{r'}\leq a^r+b^{r'}, \quad r'=\dfrac{r}{r-1}.
\]
In particular, for $a=1$, $b=c^{s}$ with $c>0$ and $r>s>1$ we have
$
c^{s}\leq 1+c^{r}
$.

\section{Evolution $p(z)$-Laplace equation}
In this section we study the problem

\label{sec:p-Laplace}
\begin{equation}
\label{eq:main}
\begin{cases}
& u_t-\operatorname{div}\left(|\nabla u|^{p(z)-2}\nabla u\right)=f(z)\quad \text{in $Q=\Omega\times (0,T)$},
\\
& \text{$u=0$ on $\partial \Omega\times (0,T)$},
\\
& \text{$u(x,0)=u_0(x)$ in $\Omega$},
\end{cases}
\end{equation}
where $\Omega\subset\mathbb{R}^n$, $n\geq 2$, is a domain with the boundary $\partial\Omega\in C^2$, the height $T$ of the cylinder $Q$ is an arbitrary finite number, $f(z)$ and $p(z)$ are given functions whose properties will be specified later.

\subsection{Weak and strong solutions}
\begin{definition}
\label{def:strong-sol}
A function $u:Q\mapsto \mathbb{R}$ is called \textbf{weak solution} of problem \eqref{eq:main} if

\begin{itemize}
\item[a)] $u\in C^{0}([0,T];L^2(\Omega))\cap \mathbf{W}_p(Q)$,  $u_t\in (\mathbf{W}_p(Q))'$;

    \item[b)] for every $\phi\in \mathbf{W}_p(Q)$,

    \begin{equation}
    \label{eq:def}
\int_Q\left(u_t\phi+|\nabla u|^{p-2}\nabla u\cdot \nabla \phi\right)\,dz=\int_Qf\phi\,dz;
\end{equation}
\item[c)] for every $\phi\in L^2(\Omega)$

    \[
    \int_{\Omega}(u(x,t)-u_0)\phi(x)\,dx\to 0\quad\text{as $t\to 0^+$};
    \]
    \item[d)] a weak solution is called \textbf{strong solution} if

        \[
        u_t\in L^{2}(Q),\qquad |\nabla u|^{p(z)}\in  L^{\infty}(0,T;L^{1}(\Omega)).
        \]
\end{itemize}
\end{definition}

We will need the following known assertions on the existence and uniqueness of weak and strong solutions to problem \eqref{eq:main}

\begin{theorem}[\cite{AntShm-2009,DNR-2012}]
\label{th:weak-sol}
Let $\partial \Omega\in {\rm Lip}$, $p(z)\in C_{{\rm log}}(\overline{Q})$, and

\begin{equation}
\label{eq:exponent-0}
\dfrac{2n}{n+2}<p^-\leq p(z)\leq p^+<\infty,\quad p^\pm=const.
\end{equation}
Then for every $u_0\in L^{2}(\Omega)$ and $f\in L^{2}(Q)$ problem \eqref{eq:main} has a unique weak solution such that

\begin{equation}
\label{eq:energy-weak}
\operatorname{ess}\sup_{(0,T)}\|u(t)\|_{2,\Omega}^2+\int_{Q_T}|\nabla u|^{p(z)}\,dz\leq C\left(\|u_0\|_{2,\Omega}^2 +\|f\|_{2,Q}^2\right)
\end{equation}
with a constant $C$ depending on $n$ and $T$.
\end{theorem}

\begin{theorem}[\cite{ArShm-2021}]
\label{th:higher-integr}
Let $f\in L^{2}(0,T;W^{1,2}_0(\Omega))$ and $u_0\in W^{1,s(\cdot)}_0(\Omega)$ with $s(x)=\max\{2,p(x,0)\}$. Assume that $\partial \Omega\in C^2$ and

\begin{equation}
\label{eq:exponent}
\operatorname{ess}\sup_{Q}|\nabla p|+\operatorname{ess}\sup_{Q}|p_t|=L<\infty.
\end{equation}
Then the weak solution $u(z)$ of problem \eqref{eq:main} is the strong solution and the following estimate holds:

\begin{equation}
\label{eq:energy-strong}
\|u_t\|_{2,Q}^2+\operatorname{ess}\sup_{(0,T)}\int_{\Omega}|\nabla u|^{s(z)}\,dz\leq C, \quad s(z)=\max\{2,p(z)\},
\end{equation}
with $C=C\left(n,T,\partial \Omega,p^\pm,\|u_0\|_{W^{1,s(\cdot)}_0(\Omega)},\|f\|_{L^2(0,T;W^{1,2}_{0}(\Omega))}\right)$.
Moreover, for every

\[
\delta\in (0,r^\ast),\qquad r^\ast=\dfrac{4p^-}{p^-(n+2)+2n},
\]
there is a constant $C$, depending on the same quantities as the constant in \eqref{eq:energy-strong} and $\delta$, such that
\begin{equation}
\label{eq:higher-integr}
\int_Q |\nabla u|^{p(z)+\delta}\,dz\leq C.
\end{equation}
\end{theorem}

The constants in \eqref{eq:energy-strong},  \eqref{eq:higher-integr} depend on $\partial\Omega$ through the norm of the parametrization of $\partial\Omega$ in the local coordinate system.

\subsection{Continuity with respect to the data}
Let us assume that we are given two sets of data:

\begin{equation}
\label{eq:assump-weak-strong}
\begin{split}
& \text{$\mathcal{S}_1\equiv \{p(z), f_1(z), u_0(x)\}$ satisfy the conditions of Theorem \ref{th:higher-integr}},
\\
& \text{$\mathcal{S}_2\equiv \{q(z), f_2(z), v_0(x)\}$ satisfy the conditions of Theorem \ref{th:weak-sol}}.
\end{split}
\end{equation}
By Theorems \ref{th:weak-sol}, \ref{th:higher-integr}, the corresponding solutions $u(z)$ and $v(z)$ are the unique strong and weak solutions of problem \eqref{eq:main} in the sense of Definition \ref{def:strong-sol}.

Let us assume that there exists a constant $\lambda\in (0,r^\ast)$ such that

\begin{equation}
\label{eq:proximity}
-(q(z)-1)(r^\ast-\lambda)\leq q(z)-p(z)\leq r^\ast-\lambda\qquad \forall z\in \overline Q.
\end{equation}
Accept the notation

\[
\begin{split}
&
\mathcal{R}=\|u_0-v_0\|^2_{2,\Omega}+\|f_1-f_2\|_{2,Q}^2+\sup_{Q}|p-q|^{q'}.
\end{split}
\]

The main result is the estimate on the difference between the strong and the weak solutions in the norm of $\mathbf{W}_{q(\cdot)}(Q)$.
\begin{theorem}
\label{th:stab-1}
Assume that $\partial\Omega\in C^2$ and conditions \eqref{eq:assump-weak-strong}, \eqref{eq:proximity} are fulfilled for the sets of data $\{p,f_1,u_0\}$ and $\{q,f_2,v_0\}$. Let $u(z)$, $v(z)$ be the corresponding strong and weak solutions of problem \eqref{eq:main}. Then

\begin{equation}
\label{eq:stab-1}
\operatorname{ess}\sup_{(0,T)}\|u-v\|_{2,\Omega}^{2}(t) + \int_{Q}|\nabla (u-v)|^{q(z)}\,dz \leq C\left(\mathcal{R}+\mathcal{R}^{\frac{q^+}{2}}+\mathcal{R}^{\frac{q^-}{2}}\right)
\end{equation}
with a constant $C$ depending on $T$, $n$, $\partial \Omega$, $p^\pm$, $q^\pm$, $\lambda$, $L$,
$\|f_1\|_{L^2(0,T;W^{1,2}_0(\Omega))}$, $\|f_2\|_{L^2(Q)}$, $\|u_0\|_{W^{1,s(\cdot)}_0(\Omega)}$, $\|v_0\|_{2,\Omega}$.
\end{theorem}

\begin{lemma}
\label{pro:test}
Let the conditions of Theorem \ref{th:stab-1} be fulfilled. If $u(x)$ is the strong solution of problem \eqref{eq:main}, then
\[
|\nabla u|\in L^{q(\cdot)}(Q),\qquad |\nabla u|^{p(z)-1}\in L^{q'(\cdot)}(Q).
\]
\end{lemma}

\begin{proof} By \eqref{eq:higher-integr}, $|\nabla u|\in L^{p(\cdot)+\delta}(Q)$ for every $\delta<r^\ast$. The inclusion $|\nabla u|\in L^{q(\cdot)}(Q)$ follows from the second inequality in \eqref{eq:proximity} and Young's inequality because $q(z)\leq p(z)+r^\ast-\lambda$: for $\delta=r^\ast-\lambda$
\[
\int_{Q}|\nabla u|^{q(z)}\,dz\leq \int_Q\left(1+|\nabla u|^{p(z)+\delta}\right)\,dz\leq |Q|+\int_Q|\nabla u|^{p(z)+\delta}\,dz.
\]
The inclusion $|\nabla u|^{p-1}\in L^{q'}(Q)$ is fulfilled if $(p-1)q'\leq p+r^\ast-\lambda$. This inequality can be written in the form
\[
\dfrac{p-1}{p+r^\ast-\lambda} =\dfrac{(p+r^\ast-\lambda)-(1+r^\ast-\lambda)}{p+r^\ast-\lambda} \leq 1-\dfrac{1}{q}\quad \Leftrightarrow\quad \dfrac{1}{q}\leq \dfrac{1+r^\ast-\lambda}{p+r^\ast-\lambda},
\]
which is equivalent to the first inequality in \eqref{eq:proximity}.
\end{proof}

\begin{lemma}
\label{pro:1}
Under the conditions of Theorem \ref{th:stab-1}, the following energy relation holds:

\begin{equation}
\label{eq:energy-w-s}
\begin{split}
\int_{Q}((u-v)_t(u-v))\,dz & + \int_Q\left(|\nabla u|^{p-2}\nabla u-|\nabla v|^{q-2}\nabla v\right)\cdot \nabla (u-v)\,dz
=\int_Q (f_1-f_2)(u-v)\,dz.
\end{split}
\end{equation}
\end{lemma}

\begin{proof}
By Theorem \ref{th:weak-sol} $v\in \mathbf{W}_q(Q)$, $v_t\in \mathbf{W}'_q(Q)$, by Theorem \ref{th:higher-integr} and Lemma \ref{pro:test} $u\in \mathbf{W}_s(Q)\subseteq \mathbf{W}_q(Q)$ with $s(z)=\max\{p(z),q(z)\}$, and $u_t\in L^2(Q)$. The function $u_t\in L^2(Q)$ can be identified as an element of $\mathbf{W}_{q}'(Q)$: for every $g\in \mathbf{W}_q(Q)$

\[
\left|\int_Q u_t g\,dz\right|\leq \|u_t\|_{2,Q}\|g\|_{2,Q}\leq \|u_t\|_{2,Q}\|g\|_{\mathbf{W}_q(Q)}.
\]
It follows that $u-v\in \mathbf{W}_q(Q)$, $(u-v)_t\in \mathbf{W}_q'(Q)$.
There is a sequence $\{\phi_k\}$ such that

\[
\phi_k\in C^{1}([0,T];C_0^{1}(\Omega)),\quad \text{$\phi_k\to u-v$ in $\mathbf{W}_q(Q)$},\quad \text{$(\phi_{k})_t\to (u-v)_t$ in $\mathbf{W}'_q(Q)$}.
\]
Taking $\phi_k$ for the test-function in identities \eqref{eq:def} for $u$ and $v$, combining the results, and then passing to the limit as $k\to \infty$ we obtain equality \eqref{eq:energy-w-s}.
\end{proof}

\subsection{Proof of Theorem \ref{th:stab-1}}
Integrating by parts in the first term and then rearranging the second term on the left-hand side, we rewrite \eqref{eq:energy-w-s} in the form
\begin{equation}
\label{eq:ident-1}
\begin{split}
\dfrac{1}{2}\|u-v\|_{2,\Omega}^{2}(t) & + \int_Q\left(|\nabla u|^{q-2}\nabla u-|\nabla v|^{q-2}\nabla v\right)\cdot \nabla (u-v)\,dz
\\
&
= \int_Q (f_1-f_2)(u-v)\,dz +\dfrac{1}{2}\|u_0-v_0\|^2_{2,\Omega}
+ \int_{Q}\left(|\nabla u|^{q-2}\nabla u-|\nabla u|^{p-2}\nabla u\right)\cdot \nabla (u-v)\,dz.
\end{split}
\end{equation}

Let us accept the notation

\[
\begin{split}
& Q_-=Q\cap \{z:\,q(z)<2\},\qquad Q_+=Q\setminus Q_-\equiv \{z:\,q(z)\geq 2\},
\\
& \mathcal{S}_q(u,v)=\left(|\nabla u|^{q-2}\nabla u-|\nabla v|^{q-2}\nabla v\right)\cdot \nabla (u-v),
\\
& \mathcal{I}_+=\int_{Q_+}S_q(u,v)\,dz,\quad \mathcal{I}_-=\int_{Q_-}S_q(u,v)\,dz,
\\
&
\mathcal{I}= \mathcal{I}_++\mathcal{I}_-\equiv \int_Q\left(|\nabla u|^{q-2}\nabla u-|\nabla v|^{q-2}\nabla v\right)\cdot \nabla (u-v)\,dz
\end{split}
\]
and rewrite equality \eqref{eq:ident-1} as follows:

\begin{equation}
\label{eq:ident-1-prim}
\begin{split}
\dfrac{1}{2}\|u-v\|_{2,\Omega}^{2}(t) & + \mathcal{I}= \int_Q (f_1-f_2)(u-v)\,dz +\dfrac{1}{2}\|u_0-v_0\|^2_{2,\Omega}
\\
& + \int_{Q}\left(|\nabla u|^{q-2}\nabla u-|\nabla u|^{p-2}\nabla u\right)\cdot \nabla (u-v)\,dz.
\end{split}
\end{equation}
\begin{lemma}
\label{le:est-1} Under the conditions of Theorem \ref{th:stab-1}

\begin{equation}
\label{eq:I-est}
\int_{Q} |\nabla (u-v)|^{q(z)}\,dz\leq C\left(\mathcal{I}+\mathcal{I}^{\frac{q^-}{2}} +\mathcal{I}^{\frac{q^+}{2}}\right)
\end{equation}
with a constant $C$ depending on $T$, $n$, $\partial \Omega$, $p^\pm$, $q^\pm$, $\lambda$, $L$,
$\|f_1\|_{L^2(0,T;W^{1,2}_0(\Omega))}$, $\|f_2\|_{L^2(Q)}$, $\|u_0\|_{W^{1,s(\cdot)}_0(\Omega)}$, $\|v_0\|_{2,\Omega}$.
\end{lemma}

\begin{proof} We will make use of the well-known inequality: for all $\xi,\zeta\in \mathbb{R}^n$

\begin{equation}
\label{eq:mon-1}
(|\xi|^{q-2}\xi-|\zeta|^{q-2}\zeta)\cdot (\xi-\zeta)\geq C\begin{cases}
|\xi-\zeta|^{q}, & q\geq 2,
\\
(1+|\xi|^2+|\zeta|^2)^{\frac{q-2}{2}}|\xi-\zeta|^2, & q\in (1,2).
\end{cases}
\end{equation}

By H\"older's inequality \eqref{eq:Holder} and relation \eqref{eq:mod} between the norm and the modular

\[
\begin{split}
& \int_{Q_-} |\nabla u-\nabla v|^{q(z)}\,dz
=\int_{Q_-}(1+|\nabla u|^2+|\nabla v|^2)^{q\frac{2-q}{4}}\left((1+|\nabla u|^2+|\nabla v|^2)^{\frac{q-2}{2}}|\nabla (u-v)|^2\right)^{\frac{q}{2}}\,dz
\\
& \leq 2 \max \left\{
\left(\int_{Q_-}(1+|\nabla u|^2+|\nabla v|^2)^{\frac{q-2}{2}}|\nabla (u-v)|^2\,dz\right)^{\frac{q^+}{2}}, \left( \int_{Q_-}\ldots\right)^{\frac{q^-}{2}}\right\}
\\
& \quad \times \max\left\{\left(\int_{Q_-}(1+|\nabla u|^2+|\nabla v|^2)^{\frac{q}{2}}\,dz\right)^{1-\frac{q^-}{2}},\,\left(\int_{Q_-}\ldots \right)^{1-\frac{q^+}{2}}\right\}
\equiv \mathcal{J}_-.
\end{split}
\]
By \eqref{eq:mon-1}

\[
\begin{split}
\mathcal{J}_- &
\leq C \max \left\{\mathcal{I}_-^{\frac{q^-}{2}}, \mathcal{I}_-^{\frac{q^+}{2}}\right\}
\max\left\{\left(1+\int_{Q}|\nabla u|^q\,dz+\int_{Q}|\nabla v|^q\,dz\right)^{1-\frac{q^-}{2}},\,(1+\ldots)^{1-\frac{q^+}{2}}\right\}
\\
& \leq C \max \left\{\mathcal{I}_-^{\frac{q^-}{2}}, \mathcal{I}_-^{\frac{q^+}{2}}\right\}\left(1+\int_{Q}|\nabla u|^q\,dz+\int_{Q}|\nabla v|^q\,dz\right)^{1-\frac{q^-}{2}}.
\end{split}
\]
Due to estimates \eqref{eq:energy-weak}, \eqref{eq:energy-strong} and property \eqref{eq:higher-integr} of higher integrability of the gradient for strong solutions, the second factor is estimated by a constant depending only on the data. It follows that

\[
\int_{Q_-}|\nabla u-\nabla v|^{q(z)}\,dz\leq \mathcal{J}_-\leq C\max \left\{\mathcal{I}_-^{\frac{q^-}{2}}, \mathcal{I}_-^{\frac{q^+}{2}}\right\}.
\]
The estimate

\[
\int_{Q^+}|\nabla u-\nabla v|^{q(z)}\,dz\leq C\mathcal{I}_+
\]
is an immediate consequence of \eqref{eq:mon-1}.
Gathering these inequalities we find that

\[
\begin{split}
\int_Q |\nabla (u-v)|^{q(z)}\,dz & \leq C \left(\mathcal{I}_++\mathcal{J}_-\right)
\leq C\left(\mathcal{I}_++\mathcal{I}_-^{\frac{q^-}{2}} +\mathcal{I}_-^{\frac{q^+}{2}}\right)\leq C\left(\mathcal{I}+\mathcal{I}^{\frac{q^-}{2}} +\mathcal{I}^{\frac{q^+}{2}}\right)
\end{split}
\]
with a constant $C=C\left(p^{\pm},q^{\pm},\|\nabla u\|_{q(\cdot),Q}, \|\nabla v\|_{q(\cdot),Q}\right)$. The estimate on $\|\nabla v\|_{q(\cdot),Q}$ follows from \eqref{eq:energy-weak}, $\|\nabla u\|_{q(\cdot),Q}$ is estimated in \eqref{eq:energy-weak}, \eqref{eq:energy-strong} and \eqref{eq:higher-integr}.
\end{proof}

Equality \eqref{eq:ident-1-prim} yields the inequality

\begin{equation}
\label{eq:start-1}
\begin{split}
\dfrac{1}{2} & \|u-v\|^2_{2,\Omega}(t) +\mathcal{I} \leq \dfrac{1}{2}\|u_0-v_0\|_{2,\Omega}^2 + \|f_1-f_2\|_{2,Q} \|u-v\|_{2,Q}
\\
& \qquad +\int_{Q}\left||\nabla u|^{q-2}\nabla u-|\nabla u|^{p-2}\nabla u\right| |\nabla (u-v)|\,dz
\\
& \equiv j_1+j_2+j_3.
\end{split}
\end{equation}

\begin{lemma}
\label{le:est-1b} Under the conditions of Theorem \ref{th:stab-1}

\begin{equation}
\label{eq:j-3}
|j_3|\leq C\sup_{Q}|p(z)-q(z)|^{q'(z)}
\end{equation}
with a constant $C$ depending on $\|\nabla u\|_{q(\cdot),Q}$, $\|\nabla u\|_{p(\cdot),Q}$, $\|\nabla v\|_{q(\cdot),Q}$.
\end{lemma}

\begin{proof}
By the generalized H\"older inequality

\[
\begin{split}
j_3 & = \int_{Q}\left||\nabla u|^{q-2}\nabla u-|\nabla u|^{p-2}\nabla u\right|| \nabla (u-v)|\,dz
\\
& \leq 2\left\|(|\nabla u|^{q-2}-|\nabla u|^{p-2})\nabla u\right\|_{q'(\cdot),Q}\|\nabla (u-v)\|_{q(\cdot),Q}.
\end{split}
\]
By the finite increments formula of Lagrange

\[
\begin{split}
|\nabla u|^{q-2}-|\nabla u|^{p-2} & =|\nabla u|^{-2}\int_0^1\dfrac{d}{d\theta}\left(|\nabla u|^{\theta q+(1-\theta)p}\right)\,d\theta
\\
& =|\nabla u|^{-2}\int_0^1 |\nabla u|^{p+\theta (q-p)}\,d\theta \ln |\nabla u|(q-p),
\end{split}
\]
whence

\begin{equation}
\label{eq:M-0}
\begin{split}
|(|\nabla u|^{q-2} & -|\nabla u|^{p-2})\nabla u |^{q'}
\\
&
\leq \left(\int_0^1\left(|\nabla u|^{p+\theta(q-p)-1}|\ln |\nabla u|| \right)\,d\theta\right)^{q'} |q-p|^{q'}:=\mathcal{M}.
\end{split}
\end{equation}
For the sake of definiteness, let us assume first that $q(z)\geq p(z)$. There are two possibilities: $|\nabla u|\geq 1$ and $|\nabla u|<1$.

\medskip

(a) $|\nabla u|\geq 1$. For every constant $\alpha>0$

\begin{equation}
\label{eq:M-1}
\begin{split}
\mathcal{M} & \leq \left(\int_0^1\left(|\nabla u|^{q-1+\alpha}\right)\left(|\nabla u|^{-\alpha}\ln |\nabla u| \right)\,d\theta \right)^{q'}|q-p|^{q'}
\leq C_{\alpha}|\nabla u|^{(q-1+\alpha)q'}|q-p|^{q'}
\end{split}
\end{equation}
with the constant

\[
C_\alpha=\left(\sup_{(1,\infty)}s^{-\alpha}\ln s\right)^{q'}=(\alpha {\rm e})^{-q'}\leq (\alpha {\rm e})^{-q^+/(q^+-1)}.
\]
Assumption \eqref{eq:proximity} allows one to choose $\alpha$ in such a way that
\[
(q-1+\alpha)q'\leq p+r^\ast-\frac{\lambda}{2}.
\]
It suffices to take for $\alpha$ any positive number satisfying the inequality

\[
\alpha<\dfrac{\lambda}{2}\dfrac{q^--1}{p^++r^\ast-\lambda}.
\]
Using Young's inequality and \eqref{eq:proximity}, we find that at the points where $|\nabla u|\geq 1$

\[
\mathcal{M}\leq C_\alpha\left(1+|\nabla u|^{p+r^\ast-\frac{\lambda}{2}}\right)|p-q|^{q'}
\]
with a constant $C$ independent of $u$.

\medskip

(b) $|\nabla u|<1$. For every constant $\beta>0$

\begin{equation}
\label{eq:M-2}
\begin{split}
\mathcal{M}
&
\leq |\nabla u|^{(p-1-\beta)q'}\left(|\nabla u|^{\beta}|\ln |\nabla u||\right)^{q'}|q-p|^{q'}
\leq C_\beta|\nabla u|^{(p-1-\beta)q'}|q-p|^{q'}
\end{split}
\end{equation}
with the constant

\[
C_\beta = \left(\sup_{(0,1)}\left(s^{\beta}|\ln s|\right)\right)^{q'}=(\beta {\rm e})^{-q'}\leq (\beta {\rm e})^{-q^+/(q^+-1)}.
\]
Take $0<\beta<p^--1$. By Young's inequality

\[
\mathcal{M}\leq C_\beta \left(1+|\nabla u|^{q}\right)|q-p|^{q'}
\]
with an independent of $u$ constant $C$.

\medskip

If $p(z)\geq q(z)$, we arrive at the same conclusion replacing $p$ and $q$ in the integrand of \eqref{eq:M-0}:

\[
|\nabla u|^{p+\theta(q-p)-1}|\ln |\nabla u||\leq C\begin{cases}
|\nabla u|^{p-1+\alpha} & \text{if $|\nabla u|\geq 1$},
\\
|\nabla u|^{q-1-\beta} & \text{if $|\nabla u|<1$}.
\end{cases}
\]
Estimate \eqref{eq:j-3} follows by integration of the estimates on $\mathcal{M}$ over $Q$ and by using \eqref{eq:higher-integr} and \eqref{eq:energy-weak}.
\end{proof}

Plugging \eqref{eq:j-3} into \eqref{eq:start-1} we continue \eqref{eq:start-1} as follows:

\begin{equation}
\label{eq:start-2}
\begin{split}
\dfrac{1}{2}\|u-v\|_{2,\Omega}^{2}(t) +\mathcal{I} & \leq C\left(\|u\|_{2,Q} +\|v\|_{2,Q}\right)\|f_1-f_2\|_{2,Q}
+ \dfrac{1}{2}\|u_0-v_0\|_{2,\Omega}^{2}+C\sup_{Q}|p(z)-q(z)|^{q'(z)}.
\end{split}
\end{equation}
Using the inequalities $\|u\|_{2,Q}\leq \|u\|_{\mathbf{W}_p(Q)}$, $\|v\|_{2,Q}\leq \|
v\|_{\mathbf{W}_q(Q)}$, we obtain the estimate

\begin{equation}
\label{eq:start-3}
\begin{split}
\|u-v\|_{2,\Omega}^{2}(t) +2\mathcal{I} & \leq \|u_0-v_0\|_{2,\Omega}^{2}
+ C\left(\|f_1-f_2\|_{2,Q}+\sup_{Q}|p-q|^{q'}\right)\leq C\mathcal{R}
\end{split}
\end{equation}
with a constant $C$ depending on $T$, $n$, $\partial \Omega$, $q^\pm$, $L$, $\|u\|_{\mathbf{W}_p}$, $\|v\|_{\mathbf{W}_{q}}$, where the last two quantities are estimated through the data of problems \eqref{eq:main} for $u$ and $v$.

The assertion of Theorem \ref{th:stab-1} immediately follows now from \eqref{eq:start-3} and \eqref{eq:I-est}.
\hfill $\Box$

\subsection{Compactness of families of strong solutions}
Let $\{p_k(z)\}$ be a sequence of functions satisfying the conditions

\begin{equation}
\label{eq:cond-comp}
\begin{split}
& \quad \dfrac{2n}{n+2}<p^-\leq p_k(z)\leq p^+<\infty\quad \text{in $Q$},
\\
& \quad \|\nabla p_k\|_{\infty,Q}+\|(p_{k})_t\|_{\infty,Q}=L<\infty
\end{split}
\end{equation}
with some positive constants $p^\pm$ and $L$. Let $\{u_m\}$ be the sequence of strong solutions of the problems

\begin{equation}
\label{eq:main-k}
\begin{split}
& \frac{\partial u_{k}}{\partial t}-\operatorname{div}\left(|\nabla u_k|^{p_k(z)-2}\nabla u_k\right)=f\quad \text{in $Q$},
\\
& \text{$u_k=0$ on $\partial\Omega\times (0,T)$},
\\
& \text{$u_k(x,0)=u_0(x)$ in $\Omega$}.
\end{split}
\end{equation}

\begin{theorem}
\label{th:comp-p-Laplace}
Let $\partial \Omega\in C^2$, $f\in L^2(0,T;W^{1,2}_0(\Omega))$, and $u_0\in W^{1,s}_0(\Omega)$ with $s=\max\{2,p^+\}$.

\begin{itemize}
\item[{\rm (i)}] If $p_k(z)$ satisfy conditions \eqref{eq:cond-comp}, then there exists $p^\ast(\cdot)\in C^{\alpha}(\overline{Q})$, $\alpha\in (0,1)$,  $p^-\leq p^\ast(z)\leq p^+$, such that the sequence $\{u_k\}$ is relatively compact in $\mathbf{W}_{p^\ast}(Q)$: there is a subsequence $\{p_{k_m}(z)\}$ such that

\begin{equation}
\label{eq:compact}
\|u_{k_m}-u_{k_s}\|_{2,\Omega}^2(t)+\int_Q|\nabla (u_{k_m}-u_{k_s})|^{p^\ast(z)}\,dz\to 0\quad \text{as $k_m,k_s\to \infty$.}
\end{equation}

\item[{\rm (ii)}] The sequence $\{u_{k_m}\}$ converges in $\mathbf{W}_{p^\ast}(Q)$ to the weak solution of problem \eqref{eq:main} with $p(z)\equiv p^\ast(z)$.
    \end{itemize}
\end{theorem}
\begin{proof}
\medskip

{\rm (i)} Let us fix some $\alpha\in (0,1)$. By virtue of \eqref{eq:cond-comp} the set $\{p_k\}$ is uniformly bounded and Lipschitz-equicontinuous. By the Ascoli-Arzela theorem $\{p_{k}\}$ is compact in $C^{\alpha}(\overline{\Omega})$.  Let $\{p_{k_m}\}$ be a Cauchy sequence in $C^{\alpha}(\overline{Q})$, $p_{k_m}(\cdot)\to p^\ast(\cdot)$ in $C^\alpha(\overline{Q})$. Take $p_{k_m}$, $p_{k_s}$ with $m$ and $s$ so large that $\sup_Q|p_{k_m}-p_{k_s}|<1$ and condition \eqref{eq:proximity} is fulfilled. By Theorem \ref{th:stab-1} with $p=p_{k_m}$, $q=p_{k_s}$

\begin{equation}
\label{eq:comp-1}
\begin{split}
\operatorname{ess}\sup_{(0,T)}\|u_{k_m}-u_{k_s}\|_{2,\Omega}^{2}(t) & + \int_{Q}|\nabla (u_{k_m}-u_{k_s})|^{p_{k_s}}\,dz
\leq C\left(\mathcal{R}+\mathcal{R}^{\frac{p^+}{2}}+\mathcal{R}^{\frac{p^-}{2}}\right)
\end{split}
\end{equation}
where

\[
\mathcal{R}=\sup_{Q}|p_{k_m}-p_{k_s}|^{p_{k_s}'}\leq \sup_{Q}|p_{k_m}-p_{k_s}|^{(p^-)'}.
\]
The constant $C$ in \eqref{eq:comp-1} depends on the constants $p^\pm$, $L$, $n$, $T$, the properties of $\partial \Omega$, and the norms $\|u_{k_m}\|_{\mathbf{W}_{p_{k_m}}(Q)}$, $\|u_{k_m}\|_{\mathbf{W}_{p_{k_s}}(Q)}$. Due to Theorems \ref{th:weak-sol}, \ref{th:higher-integr}, the last two quantities are estimated by a constant that depends on the data of problem \eqref{eq:main-k} but does not depend on $k_m$, $k_s$. Replacing $p_{k_m}$ and $p_{k_s}$ and repeating the same arguments we also have

\begin{equation}
\label{eq:comp-2}
\begin{split}
\operatorname{ess}\sup_{(0,T)}\|u_{k_m}-u_{k_s}\|_{2,\Omega}^{2}(t)
& + \int_{Q}|\nabla (u_{k_m}-u_{k_s})|^{p_{k_m}(z)}\,dz
\leq C\left(\mathcal{R}+\mathcal{R}^{\frac{p^+}{2}}+\mathcal{R}^{\frac{p^-}{2}}\right).
\end{split}
\end{equation}
Gathering \eqref{eq:comp-1}, \eqref{eq:comp-2} we obtain:

\[
\begin{split}
\operatorname{ess}\sup_{(0,T)}\|u_{k_m}-u_{k_s}\|_{2,\Omega}^{2}(t) & + \int_{Q}|\nabla (u_{k_m}-u_{k_s})|^{\max\{p_{k_m},p_{k_s}\}}\,dz
\leq C\left(\mathcal{R}+\mathcal{R}^{\frac{p^+}{2}}+\mathcal{R}^{\frac{p^-}{2}}\right)
\end{split}
\]
for all suffiently large $k_m,k_s$. Since $p_{k_s}\to p^\ast$ in $C^\alpha(\overline{Q})$, condition \eqref{eq:proximity} is fulfilled for $q=p^\ast$ and $p=\max\{p_{k_s},p_{k_m}\}$, whence \eqref{eq:compact}.

{\rm (ii)} Let us consider problem \eqref{eq:main} with the exponent $p^\ast(\cdot)\in C^{\alpha}(\overline{Q})$, $p^\ast=\lim p_{k_m}$, and the data $f$, $u_0$. By Theorem \ref{th:weak-sol} this problem has the unique weak solution $u^\ast\in \mathbf{W}_{p^\ast}(Q)$. By Theorem \ref{th:stab-1}, the uniform convergence $p_{k_m}\to p^\ast$ and inequality \eqref{eq:stab-1} yield

\[
\|u_{k_m}-u^\ast\|_{\mathbf{W}_{p^\ast}}\to 0\quad \text{as $k_m\to \infty$.}
\]
\end{proof}

\begin{corollary}
\label{cor:data-k}
Let $\{u_k\}$ be the sequence of solutions of problem \eqref{eq:main-k} with $f=f_k$ and $u_0=u_{0k}$. Assume that $\Omega$ and $p_k$ satisfy the conditions of Theorem \ref{th:comp-p-Laplace}, $f_k\in L^{2}(0,T;W^{1,2}_0(\Omega))$, $u_{0k}\in W^{1,s}_0(\Omega)$ with $s=\max\{2,p^+\}$, and

\[
\text{$f_k\to f$ in $L^2(Q)$},\qquad \text{$u_{0k}\to u_0$ in $L^2(\Omega)$}.
\]
Then there exists $p^\ast\in C^{\alpha}(\overline{Q})$ such that the  sequence $\{u_k\}$ is precompact in $\mathbf{W}_{p^\ast(\cdot)}(Q)$ and converges in $\mathbf{W}_{p^\ast(\cdot)}(Q)$ to the weak solution of problem \eqref{eq:main-k} with the data $\{p^\ast(z), f(z), u_0(z)\}$.
\end{corollary}

\section{Evolution $p(z)$-Laplacian with a nonlinear source}
\label{sec:pert}

Let us consider the problem with the nonlinear source

\begin{equation}
\label{eq:pert-1}
\begin{split}
& u_t-\operatorname{div}(|\nabla u|^{p(z)-2}\nabla u)=f(z,u) \quad \text{in $Q$},
\\
& \text{$u=0$ on $\partial \Omega\times (0,T)$},
\\
& \text{$u(x,0)=u_0(x)$ in $\Omega$}.
\end{split}
\end{equation}
The source $f(z,u)$ is a Carathéodory function (measurable in $z\in Q$ for a.e. $u\in \mathbb{R}$ and continuous in $u$ for a.e. $z\in Q$).

\begin{definition}
\label{def:pert}
A function $u:Q\mapsto \mathbb{R}$ is called \textbf{weak solution} of problem \eqref{eq:pert-1} if

\begin{itemize}
\item[a)] $u\in C^{0}([0,T];L^2(\Omega))\cap \mathbf{W}_p(Q)$,  $u_t\in (\mathbf{W}_p(Q))'$;

    \item[b)] for every $\phi\in \mathbf{W}_p(Q)$,

    \begin{equation}
    \label{eq:def-pert}
\int_Q\left(u_t\phi+|\nabla u|^{p-2}\nabla u\cdot \nabla \phi\right)\,dz=\int_Qf(z,u)\phi\,dz;
\end{equation}
\item[c)] for every $\phi\in L^2(\Omega)$ $(u(x,t)-u_0,\phi(x))_{2,\Omega}\to 0$ as $t\to 0^+$;

    \item[d)] a weak solution is called \textbf{strong solution} if $u_t\in L^{2}(Q)$ and $|\nabla u|^{p(z)}\in L^{\infty}(0,T;L^1(\Omega))$.
\end{itemize}
\end{definition}

\subsection{Existence of strong solutions}
Let $f$ satisfy the following growth conditions: there exist a constant $c_0\geq 0$ and a function $f_0(z)$ such that

\begin{equation}
\label{eq:source}
|f(z,u)|\leq c_0|u|^{\lambda-1}+f_0(z)\quad \text{in $Q\times \mathbb{R}$}
\end{equation}
with

\begin{equation}
\label{eq:pert-2}
\lambda=\max\{2,p^--\delta\}\geq 2, \quad\delta>0, \qquad f_0\in L^{\lambda'}(Q).
\end{equation}

\begin{theorem}[Th.4.1,~\cite{ant-shm-book-2015}]
\label{th:pert-1}
Let $\partial\Omega$ and $p(z)$ satisfy the conditions of Theorem \ref{th:weak-sol}. If the source $f(z,s)$ satisfies conditions \eqref{eq:source}, \eqref{eq:pert-2}, then for every $u_0\in L^{2}(\Omega)$ problem \eqref{eq:pert-1} has at least one weak solution in the sense of Definition \ref{def:pert}. The weak solution satisfies the estimate

\begin{equation}
\label{eq:pert-4}
\operatorname{ess}\sup_{(0,T)}\|u(t)\|^{2}_{2,\Omega}+\int_{Q}|\nabla u|^{p(z)}\,dz\leq C\left(\|u_0\|^{2}_{2,\Omega}+\|f_0\|_{\lambda',Q}+1\right)
\end{equation}
with an independent of $u$ constant $C$.
\end{theorem}
We will assume that the source $f(z,u)$ has the form

\begin{equation}
\label{eq:power}
\begin{split}
& f(z,u)=-a|u|^{\sigma(z)-2}u+f_0(z),\quad a=const\not=0,
\\
& \sigma\in C^{0}(\overline{Q}),\quad \|\nabla \sigma\|_{\infty,Q}=K<\infty.
\end{split}
\end{equation}
This function satisfies conditions \eqref{eq:source}, \eqref{eq:pert-2} with $\lambda=\sigma^+\equiv \sup_{Q}\sigma(z)$.
By Theorem \ref{th:pert-1} problem \eqref{eq:pert-1} with the nonlinear source satisfying \eqref{eq:power} has at least one weak solution.

\begin{theorem}
\label{th:strong-pert}
Let in the conditions of Theorem \ref{th:pert-1} $\partial\Omega\in C^2$, $p(z)$ satisfies conditions \eqref{eq:exponent}, $f(z,s)$ is of the form \eqref{eq:power}. If

\begin{equation}
\label{eq:lambda}
\begin{split}
& p^-\geq  2,\qquad \text{$2\leq \sigma(z)\leq  1+\dfrac{p(z)}{2}$ in $\overline{Q}$},
\\
& \text{$\displaystyle u_0\in W^{1,s(\cdot)}_0(\Omega)$ with $s=\max\{2,p(x,0)\}$ and $f_0\in L^{2}(0,T;W^{1,2}_0(\Omega))$},
\end{split}
\end{equation}
then every weak solution of problem \eqref{eq:pert-1} is a strong solution in the sense of Definition \ref{def:pert}. The strong solution satisfies estimates \eqref{eq:energy-strong} and possesses the property of higher integrability of the gradient \eqref{eq:higher-integr}:

\begin{equation}
\label{eq:pert-higher}
\int_{Q}|\nabla u|^{p(z)+\delta}\,dz\leq C,\quad \delta\in \left(0,\frac{4p^-}{p^-(n+2)+2n}\right),
\end{equation}
with a constant $C$ depending on the data of problem \eqref{eq:pert-1} and $\delta$, but independent of $u$.
\end{theorem}

\begin{proof}
A weak solution $u(z)$ of problem \eqref{eq:pert-1} can be regarded as a weak solution of problem \eqref{eq:main} with the known right-hand side

\[
F(z):=f(z,u(z)).
\]
Since $\lambda=\sigma^+\geq 2$, then $\lambda'=\dfrac{\lambda}{\lambda-1}\leq 2$ and by the Poincaré inequality

\[
f_0\in L^{2}(0,T;W^{1,2}_0(\Omega))\subset L^{2}(Q)\subseteq L^{\lambda'}(Q).
\]
To be able to apply Theorem \ref{th:higher-integr} and conclude that the weak solution $u(z)$ is a strong solution it remains to check that

\[
\phi(z,u)\equiv |u|^{\sigma(z)-2}u\in L^{2}(0,T;W^{1,2}_0(\Omega)).
\]
It is straightforward to compute

\[
\left|\nabla \phi\right|\leq (\sigma-1)|u|^{\sigma-2}|\nabla u|+|u|^{\sigma-1}|\ln |u|||\nabla \sigma|.
\]
By Young's inequality

\begin{equation}
\label{eq:phi}
\begin{split}
\|\nabla \phi\|_{2,Q}^2 & \leq (\sigma^+-1)\int_Q\left(|u|^{2(\sigma-2)\frac{p(z)}{p(z)-2}}+|\nabla u|^{p}\right)\,dz
\\
&
\quad +\int_Q |u|^{2(\sigma-1)}|\ln|u||^2|\nabla \sigma|^2\,dz
\equiv J_1+J_2+J_3.
\end{split}
\end{equation}
We will make use of the following assertion.

\begin{lemma}
\label{pro:2}
Let us assume that $\partial\Omega\in Lip$, $p(z)\in C(\overline{Q})$, $u\in L^{\infty}(0,T;L^2(\Omega))\cap \mathbf{W}_{p(\cdot)}(Q)$, and

\[
\operatorname{ess}\sup_{(0,T)}\|u(t)\|_{2,\Omega}^{2}+\int_{Q}|\nabla u|^{p(z)}\,dz\leq M.
\]
Then for every $\epsilon\in \left(0,\frac{1}{n}\right)$

\[
\|u\|_{p(\cdot)+\epsilon,Q}\leq C
\]
with a constant $C=C(M,p^\pm,n,\omega, |\Omega|,\epsilon)$, where $\omega$ is the modulus of continuity of $p(z)$, and $p^{\pm}$ are the maximum and minimum of $p(z)$ in $\overline{Q}$.
\end{lemma}

\begin{proof}
Let us take a finite cover of $Q$ by the cylinders $D_i=\Omega_i\times (\tau_i,\tau_{i}+h)$, $i=1,\ldots,K$ such that $\partial\Omega_i\in Lip$, $h>0$, $\tau_1=0$. Denote

\[
p_i^+=\sup_{D_i}p(z), \qquad p_i^-=\inf_{D_i}p(z).
\]
The uniform continuity of $p(z)$ in $\overline{Q}$ allows one to choose $\Omega_i$ and $h>0$ in such a way that

\begin{equation}
\label{eq:oscillation}
\dfrac{p_i^++\epsilon}{p^-_i}<1+\frac{2}{n} \quad \text{for $0<\epsilon<\dfrac{1}{n}$}.
\end{equation}
For every $i=1,2,\ldots,K$ there are two possibilities: $p_i^++\epsilon\leq 2$ and $p_i^++\epsilon>2$.

\medskip

a) $p^+_{i}+\epsilon\leq 2$. By Young's inequality

\[
\|u\|^2_{p(\cdot)+\epsilon,D_i}\leq C\left(1+\|u\|_{2,D_i}^{2}\right)\leq C(1+hM).
\]

\medskip

b) $p_i^++\epsilon>2$. In this case we apply the Gagliardo-Nirenberg inequality

\begin{equation}
\label{eq:G-N}
\begin{split}
\|u\|_{p_i^++\epsilon,\Omega_i}^{p_i^++\epsilon} & \leq C\left(\|\nabla u\|_{p_i^-,\Omega_i}^{\theta_i(p_i^++\epsilon)} +\|u\|_{2,\Omega_i}^{\theta_i(p_i^++\epsilon)}\right) \|u\|^{(1-\theta_i)(p_i^++\epsilon)}_{2,\Omega_i}
\\
& \leq C\left(\|\nabla u\|_{p_i^-,\Omega_i}^{\theta_i(p_i^++\epsilon)} +M^{\frac{1}{2}\theta_i(p_i^++\epsilon)}\right) M^{\frac{1}{2}(1-\theta_i)(p_i^++\epsilon)}
\\
& \leq C(M,p^\pm,n)\left(1+\|\nabla u\|_{p^-_i,\Omega_i}^{\theta_i(p_i^++\epsilon)}\right)
\end{split}
\end{equation}
with the exponent

\[
0<\theta_i=\dfrac{\frac{1}{2}- \frac{1}{p_i^++\epsilon}}{\frac{1}{2} -\frac{n-p_i^-}{np_i^-}}<\frac{p_i^-}{p_i^++\epsilon}.
\]
Notice that $\theta_i(p_i^++\epsilon)<p_i^-$ due to \eqref{eq:oscillation}. Integrating inequality \eqref{eq:G-N} in $t$ over the interval $(\tau_i,\tau_i+h)$ and applying Young's inequality we obtain:

\[
\begin{split}
\int_{D_i}|u|^{p(z)+\epsilon}\,dz & \leq h|\Omega|+\|u\|^{p_i^++\epsilon}_{p_i^++\epsilon,D_i}\leq  C(M,p^\pm,|\Omega|,n)\left(h +\int_{\tau_i}^{\tau_i+h}\|\nabla u\|_{p_i^-,\Omega_i}^{\theta_i(p_i^++\epsilon)}\,dt\right)
\\
& \leq C(M,p^\pm,h,|\Omega|,n)\left(1+\left(\int_{D_i}|\nabla u|^{p_i^-}\,dz\right)^{\frac{\theta_i(p_i^++\epsilon)}{p_i^-}}\right)
\\
& \leq C(M,p^\pm,h)\left(1+\int_{D_i}|\nabla u|^{p(z)}\,dz\right)
\\
&
\leq  C(M,p^\pm,h,|\Omega|,n)\left(1+M\right).
\end{split}
\]
Gathering the estimates for all $i=1,2,\ldots,K$ we obtain the required estimate

\[
\int_{Q}|u|^{p(z)+\epsilon}\,dz\leq \sum_{i=1}^{K}\int_{D_i}|u|^{p(z)+\epsilon}\,dz \leq C(M,p^\pm,\omega,|\Omega|,n).
\]
\end{proof}
\begin{remark} For $\epsilon=0$, a similar inequality was proven in \cite[Lemma 1.32]{ant-shm-book-2015}, see also \cite[Lemma 2.3]{Erhardt-2017}.
\end{remark}

For a weak solution of problem \eqref{eq:pert-1}, the second term $J_2$ of \eqref{eq:phi} is estimated in \eqref{eq:pert-4}. By virtue of Lemma \ref{pro:2} and \eqref{eq:pert-4},  the first term $J_1$ is bounded if

\[
2(\sigma(z)-2)\dfrac{p(z)}{p(z)-2}\leq p(z)\quad \Leftrightarrow \quad \sigma(z)\leq 1+\dfrac{p(z)}{2}.
\]
To estimate $J_3$ we proceed as in \eqref{eq:M-1}, \eqref{eq:M-2}: for every $\epsilon>0$ (small)

\[
J_3\leq CK^2\int_{Q}\left(|u|^{2(\sigma(z)-1)+\epsilon} +|u|^{2(\sigma(z)-1)-\epsilon}\right)\,dz.
\]
Applying Young's inequality and then using Lemma \ref{pro:2} and \eqref{eq:pert-4} we find that for $\epsilon<\frac{1}{n}$

\[
J_3\leq C\left(1+\int_Q|u|^{2(\sigma(z)-1)+\epsilon}\,dz\right)\leq C\left(1+\int_Q|u|^{p(z)+\epsilon}\,dz\right)\leq C
\]
with a constant $C$ depending only on the data.

The weak solution $u(z)$ of problem \eqref{eq:pert-1} can be considered now as a weak solution of problem \eqref{eq:main} with the source
\[
F(z)=f(z,u(z))\in L^2(0,T;W^{1,2}_0(\Omega)).
\]
By Theorem \ref{th:higher-integr} this problem has a unique strong solution $v(z)$ which must coincide with $u(z)$.
\end{proof}

\subsection{Equations with nonpositive  nonlinear sources}

Let $f(z,s)$ satisfy conditions \eqref{eq:power} and $\mathcal{S}_1$, $\mathcal{S}_2$ be the sets of data of problem \eqref{eq:pert-1} such that

\begin{equation}
\label{eq:data-pert}
\begin{cases}
\mathcal{S}_1 & \equiv \{p(z),\sigma(z),u_0,f_0\} \;\text{satisfy the conditions of Theorem \ref{th:strong-pert}},
\\
\mathcal{S}_2 & \equiv \{q(z),\mu(z),v_0,g_0\} \;\text{satisfy the conditions of Theorem \ref{th:pert-1} with}
\\
&  f(z,s)=-a|s|^{\mu(z)-2}s+g_0(z),\quad
\\
&
q(z)\geq 2,\quad 2\leq \mu(z)\leq 1+\dfrac{q(z)}{2}.
\end{cases}
\end{equation}

\begin{theorem}
\label{th:stab-pert}
Let $u(z)$, $v(z)$ be a strong and a weak solutions of problem \eqref{eq:pert-1} corresponding to the data sets $\mathcal{S}_1$ and $\mathcal{S}_2$. If $a\geq 0$, the data satisfy conditions \eqref{eq:data-pert}, and the exponents $p(z)$ and $q(z)$ satisfy condition \eqref{eq:proximity}, then

\[
\begin{split}
\operatorname{ess}\sup_{(0,T)} & \|u(t)-v(t)\|_{2,\Omega}^2  +\int_{Q}|\nabla (u-v)|^{q(z)}\,dz
\\
& \leq C\left(\|u_0-v_0\|_{2,\Omega}^2+\|f_0-g_0\|_{2,Q}+\sup_Q|p-q|^{q'}+\sup_Q|\sigma-\mu|\right)
\end{split}
\]
with a constant $C$ depending only on the data.
\end{theorem}

\begin{proof} It is known that problem \eqref{eq:pert-1} with the source defined by \eqref{eq:power} and the constant $a\geq 0$ has at  most one weak solution - \cite[Th.4.6]{ant-shm-book-2015}. Hence, $u(z)$ and $v(z)$ are unique as the weak solutions of problem \eqref{eq:pert-1}. Moreover, $u(z)$ is the strong solution and possesses the property of global higher integrability of the gradient \eqref{eq:pert-higher}. Using this property together with Lemma \ref{pro:test} we conclude that the function $u-v$ is an admissible test-function in the integral identities \eqref{eq:def-pert} for $u$ and $v$. Combining these identities and integrating by parts in $t$ we obtain the equality: for a.e. $t\in (0,T)$

\[
\begin{split}
\frac{1}{2} & \|u-v\|_{2,\Omega}^2(t) + \int_0^t\int_{\Omega}(|\nabla u|^{p(z)-2}\nabla u-|\nabla v|^{q(z)-2}\nabla v)\cdot\nabla (u-v))\,dz
\\
& = -a\int_0^t\int_{\Omega}\left(|u|^{\sigma(z)-2}u-|v|^{\mu(z)-2}v\right)(u-v)\,dz 
+\frac{1}{2}\|u_0-v_0\|^2_{2,\Omega}+\int_0^t\int_{\Omega}(u-v)(f_0-g_0).
\end{split}
\]
Rearranging this equality, applying \eqref{eq:mon-1}, and using the fact that the last equality holds for a.e. $t\in (0,T)$, we arrive at the inequality: for a.e. $t\in (0,T)$

\begin{equation}
\label{eq:inter}
\begin{split}
\frac{1}{2} \|u-v\|_{2,\Omega}^2(t) & + \int_{Q}(|\nabla u|^{q(z)-2}\nabla u-|\nabla v|^{q(z)-2}\nabla v)\cdot \nabla (u-v)\,dz
\\
& \leq  \int_{Q}(|\nabla u|^{q(z)-2}\nabla u-|\nabla u|^{p(z)-2}\nabla u)\cdot \nabla (u-v)\,dz
\\
&
-a\int_{Q}\left(|u|^{\sigma(z)-2}u-|u|^{\mu(z)-2}u\right)(u-v)\,dz
\\
&
+\frac{1}{2}\|u_0-v_0\|^2_{2,\Omega}+\int_Q(u-v)(f_0-g_0)\,dz
\\
& \equiv \mathcal{K}_1+\mathcal{K}_2+\mathcal{K}_3+\mathcal{K}_4.
\end{split}
\end{equation}

The term $\mathcal{K}_1$ is already estimated in Lemma \ref{le:est-1b}:

\[
\mathcal{K}_1\leq C\sup_Q|p-q|^{q'}
\]
with a constant $C$ depending only on the data. To estimate $\mathcal{K}_2$ we  fix an arbitrary $z\in Q$ and represent
\[
\begin{split}
|u|^{\sigma-2}-|u|^{\mu-2} & =\int_{0}^1\dfrac{d}{d\theta}\left(|u|^{\theta \sigma+(1-\theta)\mu-2}\right)\,d\theta
= \int_{0}^1|u|^{\theta \sigma+(1-\theta)\mu-2}\ln |u|\,d\theta (\sigma-\mu).
\end{split}
\]
For the sake of definiteness, let us assume that $\sigma(z)>\mu(z)$. For every sufficiently small $\epsilon>0$ there is a constant $C_\epsilon$ such that
\[
\begin{split}
\left||u|^{\sigma(z)-2}-|u|^{\mu(z)-2}\right||u| & \leq (\sigma(z)-\mu(z))\begin{cases}
|u|^{\sigma(z)-1}|\ln |u|| & \text{if $|u|>1$},
\\
|u|^{\mu(z)-1}|\ln |u|| & \text{if $|u|\leq 1$},
\end{cases}
\\
&
\leq C_\epsilon (\sigma(z)-\mu(z))\begin{cases}
|u|^{\sigma(z)-1+\epsilon} & \text{if $|u|>1$},
\\
|u|^{\mu(z)-1-\epsilon} & \text{if $|u|\leq 1$}.
\end{cases}
\\
& \leq C_{\epsilon} (\sigma(z)-\mu(z))\left(1+|u|^{\sigma(z)-1+\epsilon}\right).
\end{split}
\]
If $\sigma(z)\leq \mu(z)$, we obtain the same inequalities replacing $\mu$ and $\sigma$. Thus, for every $z\in Q$

\[
\left||u|^{\sigma(z)-2}-|u|^{\mu(z)-2}\right||u| \leq C|\sigma(z)-\mu(z)|\left(1+|u|^{\sigma(z)-1+\epsilon}+|u|^{\mu(z)-1+\epsilon}\right).
\]
It follows that

\[
\begin{split}
\mathcal{K}_2 & \leq a\sup_Q|\sigma-\mu|C\int_Q\left(1+|u-v|^2+ |u|^{2(\sigma-1)+2\epsilon}+|u|^{2(\mu-1)+2\epsilon}\right)\,dz
\\
& \leq a\sup_Q|\sigma-\mu|C\Big(|Q|+\|u\|^2_{2,Q}+\|v\|^2_{2,Q}
+\int_Q |u|^{2(\sigma-1)+2\epsilon}\,dz+\int_Q|u|^{2(\mu-1)+2\epsilon}\,dz\Big),
\end{split}
\]
where $\epsilon>0$ is still an arbitary small number. Lemma \ref{pro:2} and equality \eqref{eq:pert-4} for $u$ and $v$ yield that if

\[
\begin{split}
& 2(\sigma(z)-1)\leq p(z),\quad
2(\mu(z)-1)\leq q(z), \quad 0<\epsilon<\frac{1}{2n},
\end{split}
\]
then there is a constant $C$, depending only on the data of problems \eqref{eq:pert-1} for $u$ and $v$, such that

\[
\int_Q|u|^{2(\sigma(z)-1)+2\epsilon}\,dz+ \int_Q|u|^{2(\mu(z)-1)+2\epsilon}\,dz\leq C.
\]
The estimate on $\mathcal{K}_2$ is completed. The estimate on $\mathcal{K}_4$ follows from \eqref{eq:pert-4}: by H\"older's inequality

\[
\mathcal{K}_4\leq \left(\|u\|_{2,Q}+\|v\|_{2,Q}\right)\|f_0-g_0\|_{2,Q}\leq C \|f_0-g_0\|_{2,Q}.
\]
The term $\mathcal{K}_3$ does not require a special estimating.

Substituting into \eqref{eq:inter} the estimates on $\mathcal{K}_j$ and using \eqref{eq:mon-1} we finally obtain:

\[
\begin{split}
\operatorname{ess}\sup_{(0,T)} & \|u(t)-v(t)\|_{2,\Omega}^2  +\int_{Q}|\nabla (u-v)|^{q(z)}\,dz
\\
& \leq C\left(\|u_0-v_0\|_{2,\Omega}^2+\|f_0-g_0\|_{2,Q}+\sup_Q|p-q|^{q'}+\sup_Q|\sigma-\mu|\right)
\end{split}
\]
with a constant $C$ depending only on the data.
\end{proof}

\subsection{Equations with Lipschitz-continuous sources}
\label{subsec:reaction}
Let us consider the nonlinear source of the form

\begin{equation}
\label{eq:source-reaction}
f(z,s)=\phi(s)+f_0(z)
\end{equation}
with the function $\phi$ subject to the following conditions:

\begin{equation}
\label{eq:phi-cond}
\begin{split}
&
\text{$\phi(0)=0$, $\phi(s)$ is Lipshitz-continuous in $\mathbb{R}$ with the constant $D$}.
\end{split}
\end{equation}
Since
\[
|f(z,s)|\leq |f_0|+|\phi(s)-\phi(0)|\leq D|s|+|f_0|,
\]
it follows from Theorem \ref{th:pert-1} that problem \eqref{eq:pert-1} with the nonlinear source $f$ satisfying \eqref{eq:source-reaction} and \eqref{eq:phi-cond} has at least one weak solution $u(z)$. Moreover, by \cite[Th.4.7]{ant-shm-book-2015} this solution is unique.

If $\{p(z),u_0,f_0\}$ satisfy the conditions of Theorem \ref{th:strong-pert} and
\begin{equation}
\label{eq:incl-phi}
\phi(u)\in L^2(0,T;W^{1,2}_0(\Omega)),
\end{equation}
then the weak solution $u(z)$ is a strong solution. Inclusion \eqref{eq:incl-phi} immediately follows from \eqref{eq:pert-4} and \eqref{eq:phi-cond} : since $\phi(u)$ is Lipschitz-continuous, it is a.e. differentiable and
\[
\int_Q|\nabla \phi(u)|^2\,dz=\int_Q|\phi'(u)|^2|\nabla u|^{2}\,dz\leq D^2\int_Q(1+|\nabla u|^{p(z)})\,dz\leq C.
\]

Let us return to problem \eqref{eq:pert-1} with the source satisfying \eqref{eq:source-reaction}, \eqref{eq:phi-cond}.

\begin{theorem}
\label{th:stab-reaction}
Let $u(z)$, $v(z)$ be the weak solution of problem \eqref{eq:pert-1} with the source $f(z,s)$ satisfying conditions \eqref{eq:source-reaction}, \eqref{eq:phi-cond}, and the data $\mathcal{S}_1\equiv \{p(z),u_0,f_{0}(z)\}$, $\mathcal{S}_2\equiv \{q(z),v_0,g_{0}(z)\}$ that satisfy the conditions of Theorem \ref{th:higher-integr} and \ref{th:weak-sol} respectively. It $p(z)$, $q(z)$ satisfy condition \eqref{eq:proximity}, then

\[
\operatorname{ess}\sup_{(0,T)}\|u-v\|_{2,\Omega}^2(t)+\int_Q |\nabla (u-v)|^{q(z)}\,dz\leq C\left(\widetilde{\mathcal{R}}+\widetilde{\mathcal{R}}^{\frac{q^-}{2}} +\widetilde{\mathcal{R}}^{\frac{q^+}{2}}\right)
\]
where

\[
\widetilde{\mathcal{R}}:=C{\rm e}^{DT}\left(\sup_Q|p-q|^{q'}+\|f_0-g_0\|_{2,Q} +\|u_0-v_0\|_{2,\Omega}^{2}\right)
\]
and $C$ is a constant depending only on the data.
\end{theorem}

\begin{proof} Since $u(z)$ is a weak solution of problem \eqref{eq:pert-1}  and the data $\mathcal{S}_1$ satisfy the conditions of Theorem \ref{th:strong-pert}, $u(z)$ is a strong solution of problem \eqref{eq:pert-1}. We will argue as in  the proof of Theorem \ref{th:stab-pert}. Combining identities \eqref{eq:def-pert} for $u$ and $v$ with the test-function $u-v$, we arrive at the equality: for a.e. $t\in (0,T)$
\begin{equation}
\label{eq:est-1}
\begin{split}
\frac{1}{2} & \|u-v\|_{2,\Omega}^2(t) + \int_0^t\int_{\Omega}(|\nabla u|^{q(z)-2}\nabla u-|\nabla v|^{q(z)-2}\nabla v)\cdot\nabla (u-v))\,dz
\\
& = \int_0^t\int_{\Omega}(|\nabla u|^{q(z)-2}\nabla u-|\nabla u|^{p(z)-2}\nabla u)\cdot\nabla (u-v)\,dz
+\frac{1}{2}\|u_0-v_0\|^2_{2,\Omega}
\\
&
+\int_0^t\int_{\Omega}(u-v)(f_0-g_0)\,dz
+\int_{0}^t\int_{\Omega}\left(\phi(u)-\phi(v)\right)(u-v)\,dz.
\end{split}
\end{equation}
From this equality we derive the inequality
\begin{equation}
\label{eq:est-2}
\begin{split}
\frac{1}{4} & \|u-v\|_{2,\Omega}^2(t) + \dfrac{1}{2}\int_Q(|\nabla u|^{q(z)-2}\nabla u-|\nabla v|^{q(z)-2}\nabla v)\cdot\nabla (u-v)\,dz
\\
& \leq \mathcal{K}_1 +\mathcal{K}_3+\mathcal{K}_4+ \int_{0}^{t}\int_\Omega\left(\phi(u)-\phi(v)\right)(u-v)\,dz
\end{split}
\end{equation}
which holds for a.e. $t\in (0,T)$. The terms $\mathcal{K}_j$ were introduced and estimated in the proof of Theorem \ref{th:stab-pert}. It remains to estimate the last term on the right-hand side of \eqref{eq:est-2}. By the assumptions on $\phi(s)$ and the Lagrange intermediate value formula we have

\[
\left|\phi(u)-\phi(v)\right||u-v|\leq \int_{0}^{1}\left|\dfrac{d}{d\theta}\phi(\theta u+(1-\theta)v)\right|\,d\theta |u-v|\leq D(u-v)^2.
\]
Let us denote
\[
Y(t)=\int_{0}^{t}\int_{\Omega}(u-v)^{2}\,dz
\]
and write \eqref{eq:est-2} in the form

\begin{equation}
\label{eq:est-3}
Y'(t)+2\mathcal{I}\leq  D Y(t)+\mathcal{K}_1 +\mathcal{K}_3+\mathcal{K}_4.
\end{equation}
Omitting the nonnegative term $2\mathcal{I}$ and integrating the resulting inequality 
we obtain the estimate
\[
Y(t)\leq \left({\rm e}^{Dt}-1\right)\dfrac{1}{D}\left(\mathcal{K}_1 +\mathcal{K}_3+\mathcal{K}_4\right).
\]
Substitution of this inequality into \eqref{eq:est-3} gives: for a.e. $t\in (0,T)$
\[
\begin{split}
\|u-v\|_{2,\Omega}^2(t) & + 2\mathcal{I} \leq {\rm e}^{Dt}\left(\mathcal{K}_1 +\mathcal{K}_3+\mathcal{K}_4\right)
\\
& \leq C{\rm e}^{DT}\left(\sup_Q|p-q|^{q'}+\|f_0-g_0\|_{2,Q} +\|u_0-v_0\|_{2,\Omega}^{2}\right)\equiv \widetilde{\mathcal{R}}.
\end{split}
\]
To complete the proof we combine this estimate with the inequality proven in Lemma~\ref{le:est-1}.
\end{proof}

\begin{remark}
\label{rem:combined-sources}
The same arguments show that the stability result remains true for the solutions of problem \eqref{eq:pert-1} with the source $f(z,s)=-a|s|^{\sigma(z)-2}s+\phi(s)+f_0(z)$, provided that $p$, $q$, $\sigma$, $\mu$ satisfy the conditions of Theorem \ref{th:stab-pert} and $\phi(s)$ satisfies conditions \eqref{eq:phi-cond}. The only difference in the proofs consists in the presence of the term $\mathcal{K}_2$ on the right-hand side of \eqref{eq:est-2}.
\end{remark}

\subsection{Convergence of sequences of strong solutions}

\begin{lemma}
\label{le:conv-pert-1}
Let $\{u_k\}$ be the sequence of solutions of problem \eqref{eq:pert-1} with the exponents $\{p_k\}$, the initial functions $\{u_{0k}\}$ and the source terms $\{f_k\}$ of the form \eqref{eq:power}. Assume that $\{p_k,\sigma_k,u_{0k},f_{0k}\}$ satisfy the conditions of Theorem \ref{th:strong-pert}, $a\geq 0$, and

\[
\text{$u_{0k}\to u_0$ in $L^{2}(\Omega)$},\quad \text{$f_{0k}\to f_0$ in $L^2(Q)$},\quad \text{$\sigma_k(z)\to \sigma^\ast(z)$ uniformly in $\overline{Q}$}.
\]
There exists a subsequence $\{p_{k_m}\}$ and $p^\ast\in C^\alpha(\overline{Q})$ such that $p_{k_m}\to p^\ast$ in $C^\alpha(\overline{Q})$ and the sequence $\{u_{k_m}\}$ converges in $\mathbf{W}_{p^\ast}(Q)$ to the weak solution of problem \eqref{eq:pert-1} with the data $\{p^\ast,\sigma^\ast,u_0,f_0\}$.
\end{lemma}

\begin{lemma}
\label{le:conv-pert-2}
Let $\{u_k\}$ be the sequence of solutions of problem \eqref{eq:pert-1} with the exponents $\{p_k\}$, the initial functions $\{u_{0k}\}$ and the source terms $\{f_k\}$ of the form \eqref{eq:source-reaction}. Assume that $\{p_k,u_{0k},f_{0k}\}$ satisfy the conditions of Theorem \ref{th:higher-integr}, $\phi$ satisfies conditions \eqref{eq:phi-cond}, and

\[
\text{$u_{0k}\to u_0$ in $L^{2}(\Omega)$},\quad \text{$f_{0k}\to f_0$ in $L^2(Q)$}.
\]
Then the sequence $\{p_k\}$ converges, up to a subsequence, to a function $p^\ast\in C^{\alpha}(\overline{Q})$, $\alpha\in (0,1)$, and the corresponding subsequence of $\{u_k\}$ converges in $\mathbf{W}_{p^\ast}(Q)$ to the weak solution of problem \eqref{eq:pert-1} with the data $\{p^\ast,u_0,f_0\}$.
\end{lemma}

The assertions are an immediate byproduct of the continuity estimates of Theorem \ref{th:stab-pert} in the case of Lemma \ref{le:conv-pert-1} and  Theorem \ref{th:stab-reaction} in the case of Lemma \ref{le:conv-pert-2}. We omit the details of the proofs which are imitations of the proofs of Theorem \ref{th:comp-p-Laplace} and Corollary \ref{cor:data-k}.


\begin{thebibliography}{10}

\bibitem{AntShm-2009}
{\sc S.~Antontsev and S.~Shmarev}, {\em Anisotropic parabolic equations with
  variable nonlinearity}, Publ. Mat., 53 (2009), pp.~355--399.

\bibitem{ant-shm-book-2015}
{\sc S.~Antontsev and S.~Shmarev}, {\em Evolution {PDE}s with nonstandard
  growth conditions}, vol.~4 of Atlantis Studies in Differential Equations,
  Atlantis Press, Paris, 2015.
\newblock Existence, uniqueness, localization, blow-up.

\bibitem{ArShm-2021}
{\sc R.~Arora and S.~Shmarev}, {\em Strong solutions of evolution equations
  with {$p(x,t)$}-{L}aplacian: existence, global higher integrability of the
  gradients and second-order regularity}, J. Math. Anal. Appl., 493 (2021),
  pp.~124506, 31.

\bibitem{DHHR-2011}
{\sc L.~Diening, P.~Harjulehto, P.~H\"{a}st\"{o}, and M.~R{u}\v{z}i\v{c}ka},
  {\em Lebesgue and {S}obolev spaces with variable exponents}, vol.~2017 of
  Lecture Notes in Mathematics, Springer, Heidelberg, 2011.

\bibitem{DNR-2012}
{\sc L.~Diening, P.~N\"{a}gele, and M.~Ru\v{z}i\v{c}ka}, {\em Monotone operator
  theory for unsteady problems in variable exponent spaces}, Complex Var.
  Elliptic Equ., 57 (2012), pp.~1209--1231.

\bibitem{Eleuteri-2013}
{\sc M.~Eleuteri, P.~Harjulehto, and T.~Lukkari}, {\em Global regularity and
  stability of solutions to obstacle problems with nonstandard growth}, Rev.
  Mat. Complut., 26 (2013), pp.~147--181.

\bibitem{Erhardt-2017}
{\sc A.~H. Erhardt}, {\em Compact embedding for {$p(x,t)$}-{S}obolev spaces and
  existence theory to parabolic equations with {$p(x,t)$}-growth}, Rev. Mat.
  Complut., 30 (2017), pp.~35--61.

\bibitem{Erhardt-2017-stab}
{\sc A.~H. Erhardt}, {\em The stability of parabolic problems with nonstandard
  p(x, t)-growth}, Mathematics, 5 (2017).

\bibitem{Erhardt-2021}
\leavevmode\vrule height 2pt depth -1.6pt width 23pt, {\em Stability of weak
  solutions to parabolic problems with nonstandard growth and
  cross–diffusion}, Axioms, 10 (2021).

\bibitem{Kinnunen-Parviainen-2010}
{\sc J.~Kinnunen and M.~Parviainen}, {\em Stability for degenerate parabolic
  equations}, Adv. Calc. Var., 3 (2010), pp.~29--48.

\bibitem{Lukkari-Parviainen-2015}
{\sc T.~Lukkari and M.~Parviainen}, {\em Stability of degenerate parabolic
  {C}auchy problems}, Commun. Pure Appl. Anal., 14 (2015), pp.~201--216.

\bibitem{SS}
{\sc R.~A. Samprogna and J.~Simsen}, {\em Robustness with respect to exponents
  for nonautonomous reaction-diffusion equations}, Electron. J. Qual. Theory
  Differ. Equ.,  (2018), pp.~Paper No. 11, 15.

\bibitem{SSP3}
{\sc J.~Simsen, M.~S. Simsen, and M.~R.~T. Primo}, {\em Continuity of the flows
  and upper semicontinuity of global attractors for {$p_s(x)$}-{L}aplacian
  parabolic problems}, J. Math. Anal. Appl., 398 (2013), pp.~138--150.

\bibitem{SSP2}
{\sc J.~Simsen, M.~S. Simsen, and M.~R.~T. Primo}, {\em On
  {$p_s(x)$}-{L}aplacian parabolic problems with non-globally {L}ipschitz
  forcing term}, Z. Anal. Anwend., 33 (2014), pp.~447--462.

\bibitem{SSP1}
{\sc J.~Simsen, M.~S. Simsen, and M.~R.~T. Primo}, {\em Continuity of the flow
  and robustness for evolution equations with non globally {L}ipschitz forcing
  term}, S\~{a}o Paulo J. Math. Sci., 14 (2020), pp.~223--241.

\bibitem{Huashui-2021}
{\sc H.~Zhan and Z.~Feng}, {\em Existence and stability of the doubly nonlinear
  anisotropic parabolic equation}, J. Math. Anal. Appl., 497 (2021),
  pp.~124850, 22.

\end{thebibliography}
\end{document}